\documentclass{birkjour}

\usepackage{amssymb}
\usepackage{amsmath}

\newtheorem{theorem}{Theorem}[section]
\newtheorem{lemma}[theorem]{Lemma}
\newtheorem{proposition}[theorem]{Proposition}
\newtheorem{corollary}[theorem]{Corollary}
\theoremstyle{definition}
\newtheorem{definition}[theorem]{Definition}

 \numberwithin{equation}{section}

\newcommand{\tuple}[1]{\vec{#1}}

\newcommand {\bs}{\backslash}
\newcommand {\pmodels}{\models_P}
\newcommand {\nneg}{\texttt{Neg}}

\def\parts{\mathcal{P}}
\def\abelard{\texttt{Abelard}}
\def\eloise{\texttt{Eloise}}

\def\dom{\texttt{Dom}}

\def\free{\texttt{Free}}
\def\Dep{\textbf{D}}
\def\DD{\mathcal D}
\def\PFO{\textbf{PFO}}
\def\FO{\textbf{FO}}

\begin{document}

%-------------------------------------------------------------------------
% editorial commands: to be inserted by the editorial office
%
%\firstpage{1} \volume{228} \Copyrightyear{2004} \DOI{003-0001}
%
%
%\seriesextra{Just an add-on}
%\seriesextraline{This is the Concrete Title of this Book\br H.E. R and S.T.C. W, Eds.}
%
% for journals:
%
%\firstpage{1}
%\issuenumber{1}
%\Volumeandyear{1 (2004)}
%\Copyrightyear{2004}
%\DOI{003-xxxx-y}
%\Signet
%\commby{inhouse}
%\submitted{March 14, 2003}
%\received{March 16, 2000}
%\revised{June 1, 2000}
%\accepted{July 22, 2000}
%
%
%
%---------------------------------------------------------------------------
%Insert here the title, affiliations and abstract:
%

%\title[An Example for birkjour]
% {An Example for the Usage of the\\  birkjour Class File}

\title[Paraconsistent Team Semantics]{Dialetheism, Game Theoretic Semantics, and Paraconsistent Team Semantics}

%----------Author 1
\author[P. Galliani]{Pietro Galliani}

\address{%A
Clausthal University of Technology\\
Am Regenbogen 15\\
38678 Clausthal Zellerfeld\\
Germany}

\email{pgallian@gmail.com}

\thanks{The author acknowledges the support of the Deutsche Forschungsgemeinschaft (project number DI 561/6-1).}
%----------Author 2


%\date{March 5, 2004}
%----------additions
%\dedicatory{To my boss}
%%% ----------------------------------------------------------------------

\begin{abstract}
We introduce a variant of Dependence Logic in which truth is defined not in terms of existence of winning strategies for the Proponent ($\eloise$) in a semantic game, but in terms of lack of winning strategies for the Opponent ($\abelard$). We show that this language is a conservative but paraconsistent extension of First Order Logic, that its validity problem can be reduced to that of First Order Logic, that it capable of expressing its own truth and validity predicates, and that it is expressively equivalent to Universal Second Order Logic $\Pi_1^1$. Furthermore, we prove that a Paraconsistent Non-dependence Logic formula is consistent if and only if it is equivalent to some First Order Logic sentence; and we show that, on the other hand, all Paraconsistent \emph{Dependence} Logic sentences are equivalent to some First Order sentence with respect to truth (but not necessarily with respect to falsity).
\end{abstract}

%%% ----------------------------------------------------------------------
\maketitle
%%% ----------------------------------------------------------------------
%\tableofcontents

\section{Introduction}
Dialetheism is the philosophical position according to which there exist true contradictions. It has been proposed as a possible answer to the paradoxes of self-reference; and moreover, although this is of less relevance for the purposes of this work, it also is one potential approach to the logical treatment of vagueness and transition states.\footnote{We refer to \cite{beziau00,priest04,priest06,priest08} for a more in-depth discussion of this interesting area of research.}

Even though dialetheism is not the only possible motivation for the study of paraconsistent logics, that is, logics in which the principle of explosion (``\emph{Ex contradictione quodlibet}'': from a contradiction, everything can be derived) fails, it is easy to see that dialetheism plus the principle of explosion plus modus ponens implies triviality: if there are true contradictions, and any contradiction implies that everything is true, it follows at once that everything is true.

This is a conclusion that appears to have some practical drawbacks, at least if one maintains that what matters in a logic is the notion of truth and/or logical consequence that it generates; and hence, it comes to little surprise that dialetheism has been generally associated with paraconsistent logics.

The purpose of the present work is to draw attention to a connection between logics of imperfect information, and in particular Dependence Logic, and dialetheism: more specifically, what we aim to show is that by making a simple modification to the semantics of Dependence Logic - in brief, by defining truth in terms of the \emph{absence} of winning strategies for the Opponent in the semantic game, rather than in terms of the \emph{existence} of winning strategies for the Proponent - we obtain a paraconsistent logic with \emph{truth gluts}, that is, sentences that are both true and false. This logic is capable of expressing its own truth and validity predicates; and, differently from other logics for dialetheism such as the Logic of Paradox \cite{priest79}, its definition requires no alterations to the definition of first order model or to the interpretations of first order sentences.\footnote{These modifications may well be opportune if one wishes to model phenomena, such as vagueness, that are not easily represented in a standard first order formalism: but they should not be required for obtaining a conservative, paraconsistent extension of First Order Logic which is able to deal with the paradoxes of self-reference.}

Furthermore, as we will see, this logic has very natural properties: it is expressively equivalent to Universal Second Order Logic $\Pi_1^1$, and therefore the problem of finding its valid formulas is no harder than in the first order case, and moreover a sentence of this logic is consistent over all models if and only if it is equivalent to some first order sentence - in other words, the consistent kernel of our logic is precisely First Order Logic. 

This strongly suggests that the formalism which we will develop in this work - and for whose semantics we will give a justification which is entirely independent from issues of self-reference or paraconsistency - is, in a way, a natural extension of First Order Logic to a formalism capable of expressing its own truth predicate. 
\section{Games and Logics}
\label{sect:GTS+DL}
In this section we will recall the basic definitions of Game Theoretic Semantics and of logics of imperfect information such as IF Logic and, most importantly for our purposes, Dependence Logic. Then we will discuss some of the main properties of Dependence Logic, setting the background for the variant of it which we will introduce and discuss in Section \ref{sect:PDL}.

\subsection{Game Theoretic Semantics}
\emph{Game Theoretic Semantics} \cite{hintikka68,hintikka83, hintikkasandu97} is an approach to semantics which is inspired to Wittgenstein's language games. In brief, Game Theoretic Semantics associates sentences to \emph{semantic games}, which represent debates between a \emph{Proponent} (called usually $\eloise$), who attempts to show that the sentence is true, and an \emph{Opponent} (called usually $\abelard$), who attempts to show that the sentence is false; and then truth is defined in terms of the existence of winning strategies in these games.

In Subsection \ref{subsect:gts}, we will show how to construct such games for an extension of First Order Logic; but informally speaking, the semantic games associated to first order sentences can be defined as follows:
\begin{enumerate}
\item A first order atom $R\tuple t$ or $t_1 = t_2$ corresponds to a \emph{basic assertion}, that can be verified or falsified directly by the players: and the resulting game, in which $\eloise$ argues that $R\tuple t$ (or $t_1 = t_2$) is true in the model and $\abelard$ argues the contrary, is won by $\eloise$ if this statement holds in the model and by $\abelard$ otherwise. 
\item For $\eloise$ to assert that $\lnot \psi$ means that she is willing to assert that her opponent $\abelard$ could not argue convincingly that $\psi$ is true. Hence, the semantic game for $\lnot \psi$ is just the game for $\psi$, but with the roles of the players swapped. 
\item For $\eloise$ to assert that $\psi \vee \theta$ means that she is willing to assert that $\psi$ holds or that $\theta$ holds. Hence, $\psi \vee \theta$ corresponds to the game in which $\eloise$ chooses between $\psi$ and $\theta$, and then plays the corresponding game. 
\item For $\eloise$ to assert that $\exists v \psi$ means that she is willing to assert that $\psi[m/v]$ for at least one element $m$, where $\psi[m/v]$ is the expression obtained by substituting $m$ for $v$ in $\psi$. Hence, $\exists v \psi(v)$ corresponds to the game in which $\eloise$ picks an element $m$ of the model and plays the semantic game for $\psi[m/v]$. 
\end{enumerate}
This is all that there is in the definition of the semantic games for first order expressions. These rules are entirely natural, and, one might argue, at least as straightforward as the rules of Tarski's semantics. Truth and falsity are then defined in terms of the existence of winning strategies for the players: 
\begin{definition}
For any model $M$ and first order sentence $\phi$, we say that $\phi$ is \emph{game-theoretically true} in $M$ and we write $M \models_G^+ \phi$ if and only if $\eloise$ has a winning strategy in the game corresponding to $\phi$ in $M$; and we say that it is \emph{game-theoretically false}, and we write $M \models_G^- \phi$, if and only if $\abelard$ has a winning strategy in this game. 
\end{definition}

Now, since these games are of finite length, of perfect information (in the sense that at every stage of the game the players know everything about the history of the game) and zero-sum, it follows from Zermelo's Theorem that any such game is \emph{determined}, that is, at least one of the players has a winning strategy. Hence, we have that the Game Theoretic Semantics satisfies the Law of the Excluded Middle: 
\begin{proposition}
For any model $M$ and first order sentence $\phi$, either $M \models^+_G \phi$ or $M \models^-_G \phi$. 
\end{proposition}
\begin{corollary}
For all models $M$ and first order sentences $\phi$, $M \models^+_G \phi \vee \lnot \phi$
\end{corollary}
\begin{proof}
Suppose that $\eloise$ has a winning strategy for the game corresponding to $\phi$. Then she has a winning strategy for the game corresponding to $\phi \vee \lnot \phi$: she simply picks the left disjunct, then plays as in her winning strategy for this subgame. 

Suppose instead that $\eloise$ has no winning strategies for $\phi$. Then $\abelard$ has a winning strategy for $\phi$: and since $\lnot \phi$ corresponds to the game for $\phi$, but with the roles of the players swapped, it follows at once that $\eloise$ has a winning strategy for $\lnot \phi$. Then she also has a winning strategy for $\phi \vee \lnot \phi$, in which she picks the right disjunct and then plays her winning strategy for this subgame.
\end{proof}
Also, since swapping the players twice is the same as keeping them unvaried, it is easy to see that the law of double negation holds for this semantics:
\begin{proposition}
For all $M$, $\phi$ and $d \in \{+, -\}$, $M \models^d_G \lnot \lnot (\phi)$ if and only if $M \models^d_G \phi$. 
\end{proposition}
Furthermore, no terminal position of our games is winning for $\abelard$ and for $\eloise$ both: as Avicenna, who famously stated that 
\begin{quote}
\emph{anyone who denies the law of non-contradiction should be beaten and burned until he admits that to be beaten is not the same as not to be beaten, and to be burned is not the same as not to be burned}
\end{quote}
 would perhaps be glad to see, no reconciliation is possible, even in principle, between somebody who asserts that an atomic formula $\texttt{BEATEN}(x)$ is true and somebody who asserts that it is false. And since all possible debates between $\eloise$ and $\abelard$ end in terminal positions of this kind, we have that 
\begin{proposition}
\label{propo:lnc_ext}
For no model $M$ and first order sentence $\phi$ it holds that $M \models^+_G \phi$ and $M \models^-_G \phi$. 
\end{proposition}
We can even show that the law of the excluded middle holds as an expression of the language:
\begin{proposition}
\label{propo:lnc_int}
For all models $M$ and first order sentences $\phi$, $M \models^+_G \lnot(\phi \wedge \lnot \phi)$, where $\phi \wedge \lnot \phi$ stands for $\lnot (\lnot \phi \vee \lnot \lnot \phi)$.
\end{proposition}
\begin{proof}
By the law of double negation, we know that $M \models^+_G \lnot \lnot (\lnot \phi \vee \lnot \lnot \phi)$ if and only if $M \models^+_G \lnot \phi \vee \lnot \lnot \phi$. But this is an instance of the law of the excluded middle, which we already proved. 
\end{proof}
It is worth remarking here that the proof of Proposition \ref{propo:lnc_int} makes no use whatsoever of Proposition \ref{propo:lnc_ext}, nor vice versa. In effect, the two propositions are not related: \ref{propo:lnc_ext} says that it is never the case for any semantic game that $\abelard$ and $\eloise$ both have a winning strategy, whereas \ref{propo:lnc_int} says that $\eloise$ has a winning strategy for all semantic games of a certain kind. In the case of First Order Logic, these two statements turn out to be both true; but as we will see in Subsection \ref{subsect:prop}, there is a known conservative extension of this logic for which \ref{propo:lnc_int} fails and \ref{propo:lnc_ext} holds, and in Section \ref{sect:PDL} we will develop one for which the opposite is the case.

So far so good. But an obvious question to ask at this point is: what exactly is the relationship between Game Theoretic Semantics and Tarski's Compositional Semantics? As the following result, with which we end this subsection, shows, these two semantics are not in competition with each other: rather, they express \emph{precisely} the same truth and satisfaction conditions, albeit with different justifications! 
\begin{theorem}
For all models $M$, first order formulas $\psi$ with free variables in $\tuple v = v_1 \ldots v_n$ and assignments $s$ with domain $\{v_1 \ldots v_n\}$, $M \models_s \psi$ according to Tarski's Semantics if and only if $M \models^+_G \psi[s(\tuple v)/\tuple v]$ according to Game Theoretic Semantics. 
\end{theorem}
\begin{proof}
The proof is by structural induction on $\psi$ and poses no difficulties. 
\end{proof}
Hence, any argument that can be used for justifying the acceptability of Game Theoretic Semantics for First Order Logic can also be used to justify the acceptability of Tarski's Semantics, and vice versa. These two semantics are simply different ways of generating the same notion of truth: choosing to rely on one or on the other is at most a matter of mere convenience or of preference.

However, as we will see in the next subsections, things change once we start considering certain variants and extensions of First Order Logic. 
\subsection{Logics of Imperfect Information}
One aspect of First Order Logic which accounts for much of its expressive power is the fact that this formalism permits \emph{nested quantification}, and, in particular, \emph{alternation} between existential and universal quantifiers. Through this device, it is possible to specify complex patterns of \emph{dependence and independence} between variables: for example, in the sentence 
\begin{equation}
	\label{eq:ex1}
	\forall x \exists y \forall z \exists w R(x, y, z, w),
\end{equation}
the existential variable $w$ is a function of both $x$ and $z$, whereas the existential variable $y$ is a function of $x$ alone. 

As Skolem's normal form for (\ref{eq:ex1}) illustrates, nested quantification can be understood as a restricted form of second-order existential quantification: indeed, the above sentence can be seen to be equivalent to 
\begin{equation}
\label{eq:sk1}
\exists f \exists g \forall x \forall z R(x, f(x), z, g(x,z)).
\end{equation}
In First Order Logic, the notion of quantifier dependence or independence is intrinsically tied to the notion of \emph{scope}: an existential quantifier $\exists y$ \emph{depends} on an universal quantifier $\forall x$ if and only if the former is in the scope of the latter. As observed by Henkin in \cite{henkin61}, these patterns can be made more general. In particular, one may consider \emph{branching quantifier expressions} of the form 
\begin{equation}
\label{eq:genquant}
\left(
	\begin{array}{c c c}
		Q_{11} x_{11} & \ldots & Q_{1m}  x_{1m}\\
		& \ldots\\
		Q_{n1} x_{n1} & \ldots & Q_{nm} x_{nm}
	\end{array}
\right),
\end{equation}
where each $Q_{ij}$ is $\exists$ or $\forall$ and all $x_{ij}$ are distinct. The intended interpretation of such an expression is that each $x_{ij}$ may depend on all $x_{ij'}$ for $j' < j$, but \emph{not} on any $x_{i'j'}$ for $i' \not = i$: for example, in the sentence
\begin{equation}
\label{eq:ex2}
\left(
	\begin{array}{c c}
		\forall x & \exists y\\
		\forall z & \exists w
	\end{array}
\right)
R (x, y, z, w)
\end{equation}
the variable $y$ depends on $x$ but not on $z$, and the variable $w$ depends on $z$ but not on $x$, and hence the corresponding Skolem expression is
\begin{equation}
\label{eq:sk2}
\exists f \exists g \forall x \forall z R(x, f(x), z, g(z))
\end{equation}
If, as we said, quantifier alternation in First Order Logic can be understood as a restricted form of second order existential quantification, then, as a comparison between (\ref{eq:sk1}) and (\ref{eq:sk2}) makes clear, allowing branching quantifiers can be understood as a weakening of these restrictions.

How restricted is second order existential quantification in Branching Quantifier Logic, that is, in First Order Logic extended with branching quantifier?

As proved by Enderton and Walkoe in \cite{enderton70} and \cite{walkoe70}, the answer is \emph{not restricted at all}! Branching Quantifier Logic is precisely as expressive as Existential Second Order Logic ($\Sigma_1^1$). Hence, Branching Quantifier Logic can be understood as an alternative approach to the study of $\Sigma_1^1$, of its fragments and of its extensions; and indeed, much of the research done on the subject can be seen as an attempt to study $\Sigma_1^1$ through the lens of these variants of first-order logic.

Independence Friendly Logic \cite{hintikkasandu89,hintikka96,mann11}, also called IF Logic, can be thought of as a linearization of Branching Quantifier Logic: rather than dealing the unwieldy quantifier matrices of (\ref{eq:genquant}), Hintikka and Sandu introduced \emph{slashed quantifiers} $\exists v / V$ with the intended interpretation of ``there exists a $v$, chosen \emph{independently} from the variables in $V$''. For example, the sentence (\ref{eq:ex2}) can be translated in IF Logic as 
\begin{equation}
\label{eq:ex3}
\forall x \exists y \forall z (\exists w / \{x, y\}) R(x, y, z, w)
\end{equation}
This - at first sight entirely unproblematic - modification led to a number of important innovations on the semantic side. 

Game-theoretical explanations for the semantics of branching quantifiers predate the development of IF Logic; but it is with IF Logic that the Game Theoretic Semantics for First Order Logic was extended and adapted to a logic of imperfect information in a formal way. In Subsection \ref{subsect:gts}, we will present in some detail a successor of the Game Theoretic Semantics for IF Logic; but in brief, in the Game Theoretic Semantics for IF Logic, slashed quantifiers correspond to \emph{imperfect information moves} in which the corresponding player has to select a value for the quantified variable \emph{without} having access to the values of the slashed variables. 

Dependence Logic \cite{vaananen07,vaananen10}, which we will start discussing in the next subsection, is another variant of First Order Logic which adds to it the possibility of expressing more patterns of dependence and independence between variables. Rather than doing so by adding new forms of quantification, as we will see, Dependence Logic isolates the notion of dependence in some new \emph{dependence atoms}: this approach - which, as discussed in \cite{vaananen07}, is equivalent to the one of IF Logic, in the sense that any IF Logic sentence is equivalent to a Dependence Logic sentence and vice versa - proved itself most fruitful, and sparked the development of a number of technical results about the properties of this formalism and its extensions and variants.\footnote{We cannot give here a comprehensive account of this rapidly growing area of research: but as a possible starting point, and with no pretense of completeness, we can refer the interested reader to \cite{vaananen07,vaananen07b,abramsky09,kontinenv09,kontinennu09,kontinenv10,gradel13,durand11,engstrom12,galliani12,galliani13}.}
\subsection{Dependence Logic: Syntax}
Dependence Logic is a conservative extension of First Order Logic which adds to it \emph{dependence atoms} $=\!\!(\tuple t_1, \tuple t_2)$, with the intended interpretation of ``the value of $\tuple t_2$ is a function of the value of $\tuple t_1$'':\footnote{In some versions of Dependence Logic, only atoms of the form $=\!\!(\tuple t_1, t')$ are allowed, where $t'$ is a single term. This makes no real difference, as $=\!\!(\tuple t_1, \tuple t_2)$ is easily seen to be equivalent to $\bigwedge_{t' \in \tuple t_2} =\!\!(\tuple t_1, t')$.} for example, Equations (\ref{eq:ex2}) and (\ref{eq:ex3}) would be translated in Dependence Logic as 
\begin{equation}
\forall x \exists y \forall z \exists w(=\!\!(z, w) \wedge R(x, y, z, w))
\end{equation}
where the fact that the choice of $w$ may depend only on $z$ is represented by $=\!\!(z,w)$. 

More formally, the syntax of Dependence Logic can be given as follows:
\begin{definition}[Syntax]
Let $\Sigma$ be a first order signature. Then the set of all Dependence Logic formulas over $\Sigma$ is given by 
\[
	\psi ::= R \tuple t ~|~  t_1 = t_2 ~|~ =\!\!(\tuple t_1, \tuple t_2) ~|~ \lnot \psi ~|~ \psi \vee \psi ~|~ \exists v \psi
\]
where $R$ ranges over all relation symbols in $\Sigma$, $\tuple t$ ranges over all tuples of terms of signature $\Sigma$ and of the required arities, $t_1$ and $t_2$ range over terms, $\tuple t_1$ and $\tuple t_2$ range over all tuples of terms and $v$ ranges over all variables. 
\end{definition}
As usual, we will write $t_1 \not = t_2$ as a shorthand for $\lnot (t_1 = t_2)$, $\psi \wedge \theta$ as a shorthand for $\lnot ((\lnot \psi) \vee (\lnot \theta))$ and $\forall v \psi$ as a shorthand for $\lnot (\exists v (\lnot \psi))$. 

\begin{definition}[Free Variables and Sentences]
Let $\psi$ be a Dependence Logic formula over a signature $\Sigma$. Then the set $\free(\psi)$ of the \emph{free variables} of $\psi$ is defined inductively as follows:
\begin{enumerate}
\item For all atomic formulas $R \tuple t$, $\free(R \tuple t)$ is the set of all variables occurring in $\tuple t$;
\item $\free(t_1 = t_2)$ is the set of all variables occurring in $t_1$ and $t_2$;
\item $\free(=\!\!(\tuple t_1, \tuple t_2))$ is the set of all variables occurring in $\tuple t_1$ and $\tuple t_2$;
\item $\free(\lnot \psi) = \free(\psi)$;
\item $\free(\psi \vee \theta) = \free(\psi) \cup \free(\theta)$;
\item $\free(\exists v\psi) = \free(\psi) \bs \{v\}$.
\end{enumerate}
A formula $\psi$ is said to be a \emph{sentence} if $\free(\psi) = \emptyset$.
\end{definition}
\subsection{Dependence Logic: Game Theoretic Semantics}
\label{subsect:gts}
In this section, we will briefly recall the game theoretic semantics for Dependence Logic. Our treatment will be entirely similar to the one of \cite{vaananen07}, and we refer to it for a more detailed exposition.

As in the case of the Game Theoretic Semantics of First Order Logic, our games will involve two players, the existential one (called \eloise) who attempts to demonstrate the truth of a sentence in a model, and the universal one (called \abelard) who attempts to demonstrate its falsity. However, differently from the case of First Order Logic, our games will be \emph{games of imperfect information}: 
\begin{definition}[Semantic Games]
Let $M$ be a first order model and let $\phi$ be a Dependence Logic sentence over the signature of $M$. Then the \emph{semantic game} $G^M(\phi)$ is defined as follows: 
\begin{itemize}
\item Positions are triples $(\psi, s, \alpha)$, where $\psi$ is an instance\footnote{Different instances of a subformula inside of a formula - for example, the two occurrences of $Px$ in $Px \vee Px$ - correspond to different positions of our games. With an abuse of notation, we will write instances of subformulas as if they were simply subformulas.} of a subformula of $\phi$, $s$ is a variable assignment and $\alpha \in \{\eloise, \abelard\}$;
\item The initial position is $(\phi, \emptyset, \eloise)$;
\item The rules of the game are as follows: 
\begin{enumerate}
\item If the current position is of the form $(R \tuple t, s, \alpha)$, then Player $\alpha$ wins if and only if $M \models_s R \tuple t$ in the usual first order sense. Otherwise, Player $\alpha^*$ wins, where $\eloise^* = \abelard$ and $\abelard^* = \eloise$.
\item If the current position is of the form $(t_1 = t_2, s, \alpha)$, Player $\alpha$ wins if and only if $M \models_s t_1 = t_2$ in the usual first order sense. Otherwise, Player $\alpha^*$ wins. 
\item If the current position is of the form $(=\!\!(\tuple t_1, \tuple t_2), s, \alpha)$, Player $\alpha$ wins.
\item If the current position is of the form $(\lnot \psi, s, \alpha)$, then the next position is of the form $(\psi, s, \alpha^*)$.
\item If the current position is of the form $(\psi \vee \theta, s, \alpha)$, then the next position is either $(\psi, s, \alpha)$ or $(\theta, s, \alpha)$, and Player $\alpha$ chooses. 
\item If the current position is of the form $(\exists v\psi, s, \alpha)$, Player $\alpha$ picks an element $m \in \dom(M)$. The next position is then $(\psi, s[m/v], \alpha)$.
\end{enumerate}
\end{itemize}
\end{definition}
The definitions of \emph{strategy} and \emph{play} can be given as usual, and we will not report them here. For the sake of generality we will admit \emph{nondeterministic} strategies, which allow for the selection of multiple possible successors of a given position.\footnote{This is of no consequence for case of Dependence Logic proper, but it is useful in order to preserve the property of locality (Theorem \ref{theo:local} here) in the case of other team-semantics based extensions of first-order logic. See \cite{galliani12} for details.} However, in order to represent the intended meaning of dependence atoms we need the following notion of \emph{uniform strategy}: 
\begin{definition}
A strategy $\sigma$ for Player $\alpha \in \{\abelard, \eloise\}$ is \emph{uniform} if for any two plays $\tuple p_1, \tuple p_2$ in which Player $\alpha$ follows $\sigma$ and which end with $(=\!\!(\tuple t_1, \tuple t_2), s, \alpha)$ and $(=\!\!(\tuple t_1, \tuple t_2), s', \alpha)$ respectively, for the same instance of the dependence atom,
\[
	\tuple t_1\langle s\rangle = \tuple t_1\langle s'\rangle \Rightarrow \tuple t_2\langle s \rangle = \tuple t_2\langle s'\rangle.
\]
\end{definition}
The rationale behind this condition should be clear: a strategy for $\alpha$ is uniform if and only if when Player $\alpha$ follows it and reaches a dependence atom $=\!\!(\tuple t_1, \tuple t_2)$, he or she can guarantee that the value of $\tuple t_2$ will be a function of the value of $\tuple t_1$.

Truth and falsity are then defined in terms of the existence of uniform winning strategies: 
\begin{definition}
For any sentence $\phi$ and model $M$, we say that $\phi$ is \emph{true} in $M$, and we write $M \models^+ \phi$, if and only if $\eloise$ has a uniform winning strategy in $G^M(\phi)$.

Similarly, $\phi$ is \emph{false} in $M$ if and only if $\abelard$ has a uniform winning strategy in $G^M(\phi)$; and in that case, we write $M \models^- \phi$.
\end{definition}
The following result then follows at once from the definition of the game theoretic semantics: 
\begin{proposition}
For all $M$ and $\phi$ over the signature of $M$, $\phi$ is true if and only if $\lnot \phi$ is false and $\phi$ is false if and only if $\lnot \phi$ is true. 
\end{proposition}
\begin{proof}
Suppose that $M \models^+ \phi$, that is, that $\eloise$ has a uniform winning strategy in $G^M(\phi)$. Now, $G^M(\lnot \phi)$ is played exactly like $G^M(\phi)$, except for the fact that the first move of this game exchanges the roles of the two players; and therefore, $\abelard$ has a uniform winning strategy in this game, and hence $M \models^- \lnot \phi$. 

The other direction and the second biimplication are verified similarly.
\end{proof}
The fact that we are placing restrictions on which strategies we allow accounts for the possible existence of \emph{truth gaps}. For example, consider the sentence $\forall x \exists y ( =\!\!(\emptyset, x) \wedge x = y)$. With respect to a model $M$ with domain $\{0, 1\}$, the semantic game corresponds to the \emph{matching pennies} game: 
\begin{enumerate}
\item First, $\abelard$ selects an element $m_1 \in \{0,1\}$; 
\item Then, $\eloise$ selects another element $m_2 \in \{0, 1\}$, independently on $\abelard$'s choice:\footnote{If the choice of $\eloise$ is dependent on the choice of $\abelard$ then $\eloise$'s strategy is not uniform. Note that there is no need for $\abelard$'s strategy to ever pick the left conjunct $=\!\!(\emptyset, x)$, which corresponds to a winning position for $\eloise$; the mere fact that such a choice is \emph{possible} implies that $\eloise$'s move must be independent from $\abelard$'s.}
\item Then the game ends. If $m_1 = m_2$, $\eloise$ wins; otherwise, $\abelard$ wins.
\end{enumerate}
Clearly, neither player has a winning strategy for this game: no matter what value $\abelard$ picks, he cannot make it impossible for $\eloise$ to select the same value, but on the other hand since $\eloise$ is not allowed to spy $\abelard$'s choice she will not be able to guarantee that she will copy $\abelard$'s selection.
\subsection{Dependence Logic: Team Semantics}
The game theoretic semantics which we described in the previous section is a simple generalization of the one of First Order Logic. However, as we know, First Order Logic can also be given the equivalent, compositional \emph{Tarski Semantics}, in which open formulas correspond to \emph{satisfaction conditions} over variable assignments. Is there any way to extend Tarski's Semantics to the case of Dependence Logic? 

As \cite{cameron01} shows by means of a combinatorial argument, the answer is negative, as there exist models with more non-interchangeable Dependence Logic\footnote{Cameron and Hodges' proof refers to IF Logic instead; but their argument can be adapted with no difficulty to the case of Dependence Logic.} formulas with one free variable than sets of elements.\footnote{The proof does not generalize to infinite models; however, \cite{galliani11} extends this result to this further case by introducing a notion of \emph{sensible semantics}.} However, not all is lost: in \cite{hodges97}, Hodges developed a compositional semantics for Independence Friendly Logic, which we will call ``Team Semantics'' after \cite{vaananen07}, and whose Dependence Logic variant (also from \cite{vaananen07}) we will now describe. 

The fundamental intuition behind Team Semantics is to take in consideration \emph{teams}, that is, sets of assignments, and to generalize semantic games to initial assignments picked randomly from a team: 
\begin{definition}[Team]
Let $M$ be a first order model, and let $V$ be a set of variables. A \emph{team} over $M$ with domain $V$ is a set $X$ of variable assignments over $M$ with domain $V$.
\end{definition}
\begin{definition}[$G^M_X(\psi)$ and Team Satisfaction]
Let $M$ be a first order model, let $\psi$ be a Dependence Logic formula over its signature, and let $X$ be a team over $M$ such that $\free(\psi) \subseteq \dom(X)$. Then the game $G^M_X(\psi)$ is played precisely like $G^M(\psi)$, but from a starting position selected randomly among $\{(\psi, s, \eloise) : s \in X\}$.

If $\eloise$ has a uniform winning strategy in $G^M_X(\psi)$, we say that the team $X$ \emph{satisfies positively} (or just ``satisfies'') $\psi$ in $M$ and we write $(M, X) \models^+ \psi$; and similarly, if $\abelard$ has a uniform winning strategy in $G^M_X(\psi)$ we say that $X$ \emph{satisfies negatively} $\psi$ in $M$ and we write $(M, X) \models^- \psi$.
\end{definition}
Clearly, we have that for all models $M$ and sentences $\phi$, $M \models^\pm \phi$ if and only if $(M, \{\emptyset\}) \models^\pm \phi$.\footnote{Some care must be paid to distinguish the team $\{\emptyset\}$, which contains the empty assignment, from the empty team $\emptyset$. As we said, if $\phi$ is a sentence then $M \models^+ \phi$ if and only if $(M, \{\emptyset\}) \models^+ \phi$; but on the other hand, as we will see, $(M, \emptyset) \models^\pm \phi$ for all $\phi$.}

The advantage is that this satisfaction definition is compositional: 
\begin{theorem}
For all models $M$, teams $X$, relation symbols $R$ in its signature, tuples of terms $\tuple t$, $\tuple t_1$, $\tuple t_2$ and terms $t_1$, $t_2$ in its signature, and variables $v$,
\begin{description}
\item[TS-atom$^+$] $(M, X) \models^+ R \tuple t$ if and only if for all $s \in X$, $M \models_s R \tuple t$ in the usual first order sense; 
\item[TS-eq$^+$] $(M, X) \models^+ t_1 = t_2$ if and only if for all $s \in X$, $M \models_s t_1 = t_2$ in the usual first order sense;
\item[TS-dep$^+$] $(M, X) \models^+ =\!\!(\tuple t_1, \tuple t_2)$ if and only if for all $s, s' \in X$, 
\[
	\tuple t_1\langle s\rangle = \tuple t_1\langle s' \rangle \Rightarrow \tuple t_2\langle s\rangle = \tuple t_2\langle s'\rangle;
\]
\item[TS-neg$^+$] $(M, X) \models^+ \lnot \psi$ if and only if $(M, X) \models^- \psi$;
\item[TS-$\vee^+$] $(M, X) \models^+ \psi \vee \theta$ if and only if $X = X_1 \cup X_2$ for two subteams $X_1$ and $X_2$ such that $(M, X_1) \models^+ \psi$ and $(M, X_2) \models^+ \theta$; 
\item[TS-$\exists^+$] $(M ,X) \models^+ \exists v \psi$ if and only if there exists a function $F: X \rightarrow \parts(\dom(M))\backslash \{\emptyset\}$ such that $(M, X[F/v]) \models^+ \psi$, where $X[F/v] = \{s[m/v] : s \in X, m \in F(s)\}$
\end{description}
and
\begin{description}
\item[TS-atom$^-$] $(M, X) \models^- R \tuple t$ if and only if for all $s \in X$, $M \not \models_s R \tuple t$ in the usual first order sense; 
\item[TS-eq$^-$] $(M, X) \models^- t_1 = t_2$ if and only if for all $s \in X$, $M \not \models_s t_1 = t_2$ in the usual first order sense;
\item[TS-dep$^-$] $(M, X) \models^- =\!\!(\tuple t_1, \tuple t_2)$ if and only if $X = \emptyset$;
\item[TS-neg$^-$] $(M, X) \models^- \lnot \psi$ if and only if $(M, X) \models^+ \psi$;
\item[TS-$\vee^-$] $(M, X) \models^- \psi \vee \theta$ if and only if $(M, X) \models^- \psi$ and $(M, X) \models^- \theta$;
\item[TS-$\exists^-$] $(M, X) \models^- \exists v \psi$ if and only if $(M, X[M/v]) \models^- \psi$, where $X[M/v] = \{s[m/v] : s \in X, m \in \dom(M)\}$.
\end{description}
\end{theorem}
\begin{proof}
The proof is by structural induction on the formulas and poses no particular difficulties. We refer the interested reader to \cite{hodges97,vaananen07}.
\end{proof}
It is now easy to find the rules corresponding to our derived connectives and operators: 
\begin{corollary}~\\
\begin{description}
\item[TS-$\wedge^+$] $(M, X) \models^+ \psi \wedge \theta$ if and only if $(M, X) \models^+ \psi$ and $(M, X) \models^+ \theta$; 
\item[TS-$\wedge^-$] $(M, X) \models^- \psi \wedge \theta$ if and only if $X = X_1 \cup X_2$ for two subteams $X_1, X_2$ such that $(M, X_1) \models^- \psi$ and $(M, X_2) \models^- \theta$;
\item[TS-$\forall^+$] $(M, X) \models^+ \forall v \psi$ if and only if $(M, X[M/v]) \models^+ \psi$; 
\item[TS-$\forall^-$] $(M ,X) \models^- \forall v \psi$ if and only if there exists a function $F: X \rightarrow \dom(M)$ such that $(M, X[F/v]) \models^- \psi$.
\end{description}
\end{corollary}
We conclude this subsection by observing that there is a clear asymmetry between the positive and negative conditions for dependence atoms: an atom $=\!\!(\tuple t_1, \tuple t_2)$ is true in a team $X$ if and only if the atom satisfies the corresponding dependence condition, whereas it is false if and only if $X = \emptyset$.

This is a direct consequence of the definition of truth in terms of existence of winning strategies: in brief, while a restriction on the information available to a player can affect the existence or non-existence of winning strategies for her, it has no such effect on the potential existence of winning strategies for the \emph{other} player - after all, the first player might well ``get lucky'' and select the best response even in the absence of sufficient information.
\subsection{Some Known Results}
\label{subsect:prop}
In this subsection, we will recall some of the properties of Dependence Logic. Except where stated otherwise, all results are from \cite{vaananen07}.
\begin{theorem}[Locality]
\label{theo:local}
For all formulas $\psi$ of Dependence Logic, all models $M$ and all teams $X_1$ and $X_2$ whose restrictions to $\free(\phi)$ coincide, 
\[
	(M, X_1) \models^\pm \psi \Leftrightarrow (M, X_2) \models^\pm \psi.
\]
\end{theorem}

\begin{theorem}[First Order Formulas]
If $\psi$ is first order,
\begin{description}
\item[TS-FO$^+$] $(M, X) \models^+ \psi$ if and only if for all $s \in X$, $M \models_s \psi$ in the usual first order sense; 
\item[TS-FO$^-$] $(M, X) \models^- \psi$ if and only if for all $s \in X$, $M \not \models_s \psi$ in the usual first order sense.
\end{description}
\end{theorem}
Therefore, a first order formula holds in a team only if it holds in all of its assignments, and it fails in it only if it fails for all of its assignments. This clarifies the reason for the presence of truth value gaps in this semantics: if a team contains assignments which satisfy a formula and assignments which do not satisfy it, then the formula is neither true nor false with respect to it. Hence, with respect to first order formulas, Team Semantics can be thought of as a form of \emph{supervaluationism}.

Furthermore, we have as an easy corollary of this result that 
\begin{corollary}
If $\phi$ is a first order sentence, $M \models^+ \phi$ according to the semantics of Dependence Logic if and only if $M \models \phi$ according to Tarski's Semantics. Similarly, $M \models^- \phi$ if and only if $M \not \models \phi$ according to Tarski's Semantics.
\end{corollary}

However, in general the connectives of Dependence Logic behave rather differently from those of First Order Logic. In particular, disjunction is not idempotent: indeed, consider any model $M$ with two elements $0$ and $1$ and the team $X = \{(x := 0), (x :=1)\}$. Then we have that $(M, X) \models^+ =\!\!(\emptyset, x) \vee =\!\!(\emptyset, x)$, since $X$ can be split into two subteams $X_1 = \{(x :=0)\}$ and $X_2 = \{(x := 1)\}$ with respect to which $x$ is constant; but on the other hand, $x$ is not constant in $X$, and therefore $(M, X) \not \models^+ =\!\!(\emptyset, x)$.

\begin{theorem}[Downwards Closure]
\label{theo:dc}
For all formulas $\psi$, models $M$, $d \in \{+, -\}$ and teams $X, Y$, if $(M, X) \models^d \psi$ and $Y \subseteq X$ then $(M, Y) \models^d \psi$.
\end{theorem}
The empty team $\emptyset$ satisfies all Dependence Logic formulas: 
\begin{proposition}
\label{propo:emptyset}
For all models $M$ and formulas $\phi$ over the signature of $M$, $(M, \emptyset) \models^+ \phi$ and $(M, \emptyset) \models^- \phi$. 
\end{proposition}
As a consequence of this fact, the law of non-contradiction does not hold in Dependence Logic, strictly speaking: for any formula $\phi$, $(M, \emptyset) \models^+ \phi \wedge \lnot \phi$. Hence, Dependence Logic displays a (very weak) form of paraconsistency.

However, this following variant of non-contradiction holds:
\begin{theorem}[Law of Non-Contradiction]
\label{theo:lnc}
If $(M, X) \models^+ \phi$ and \\$(M, X) \models^- \phi$ then $X = \emptyset$.
\end{theorem}
In particular, this implies that for any sentence $\phi$, it is never the case that $M \models^+ \phi$ and $M \models^- \phi$: no sentence is both true and false (although, as we saw, there exist sentences which are neither). 

Note, however, that $\lnot(\phi \wedge \lnot \phi)$ is not, in general, valid.
 For example, let $M$ contain at least two elements and take $ \phi = \forall x \exists y (=\!\!(\emptyset, y) \wedge x = y)$: as we saw, neither player has a winning strategy for $G^M(\phi)$, and it follows easily that the same holds for $G^M(\lnot (\phi \wedge \lnot \phi))$. 

The same choice of $\phi$, as an aside, also shows that $\phi \vee \lnot \phi$ is not a theorem of Dependence Logic; but due to the fact that Dependence Logic admits truth gaps, this is less surprising. 

What is, instead, highly surprising is the following result from \cite{kontinenv10}, which shows that the negation of Dependence Logic - which coincides to the usual contradictory negation over first order sentences - behaves, in general, very differently indeed from the contradictory negation:
\begin{theorem}
\label{theo:DL_neg_nonsem}
Let $\psi_1$ and $\psi_2$ be two formulas such that $(M, X) \models^+ \psi_1 \wedge \psi_2$ only if $X = \emptyset$. Then there exists a Dependence Logic expression $\theta$ such that 
\[
	(M, X) \models^+ \theta \Leftrightarrow (M, X) \models^+ \psi_1
\]
and 
\[
	(M, X) \models^- \theta \Leftrightarrow (M, X) \models^+ \psi_2
\]
for all models $M$ whose signature contains the ones of $\psi_1$ and $\psi_2$ and for all $X$ with domain containing $\free(\psi_1) \cup \free(\psi_2)$.
\end{theorem}

\begin{theorem}[Dependence Logic and $\Sigma_1^1$ - sentences]
\label{theo:DLsent}
For any sentence $\phi$ of Dependence Logic there exists a existential second order sentence $\Phi$ such that $M \models^+ \phi$ if and only if $M \models \Phi$. Conversely, any sentence $\Phi$ of Existential Second Order Logic is equivalent to some Dependence Logic sentence $\phi$.
\end{theorem}

What about open formulas of Dependence Logic? The following theorem shows that their satisfaction conditions correspond to $\Sigma_1^1$-expressible, downwards monotonic conditions over relations:
\begin{theorem}
\label{theo:DLtoSigma11}
For all formulas $\psi$ of Dependence Logic and all sets of variables\footnote{We assume that the ordering of the variables of $V$ is fixed in advance.} $V = \{v_1 \ldots v_n\}$ with $\free(\psi) \subseteq V$ there exists a $\Sigma_1^1$ sentence $\Phi(R)$, where $R$ is a new $|V|$-ary relation symbol, such that 
\[
	(M, X) \models^+ \psi \Leftrightarrow M \models \Phi(X(V))
\]
for all models $M$ and teams $X$ and for $X(V) = \{s(v_1 \ldots v_n) : s \in X\}$.

Furthermore, $R$ occurs only negatively in $\Phi(R)$. 
\end{theorem}

The converse of this result - arguably one of the most important and useful results about Dependence Logic - was proved in \cite{kontinenv09}: 
\begin{theorem}
\label{theo:Sigma11toDL}
Let $\Phi(R)$ be a $\Sigma_1^1$ sentence in which $R$ occurs only negatively and has arity $n$, and let $V = \{v_1 \ldots v_n\}$ be a set of variables. Then there exists a Dependence Logic formula $\psi$ with free variables in $V$ such that \[
	(M, X) \models^+ \psi \Leftrightarrow M \models \Phi(X(V))
\]
for all models $M$ and teams $X$.
\end{theorem}

Another one of the most interesting properties of Dependence Logic, whose proof can be found in \S 6.4 of \cite{vaananen07}, is that this logic can define its own truth operator! More precisely, what can be proved is that 
\begin{theorem}
\label{theo:DL_TP}
For any finite signature $\Sigma$ containing the binary operations $+$ and $\times$ there exists a Dependence Logic formula $\tau(x)$ with signature $\Sigma$ such that for all models $M_{\omega}$ of Peano's Axioms with signature $\Sigma$ and for all Dependence Logic sentences $\phi$, 
\[
	M_{\omega} \models^+ \phi \mbox{ if and only if } M_{\omega} \models^+ \tau(\ulcorner \phi\urcorner)
\]
where $\ulcorner \phi\urcorner$ represents a G\"odel numbering for $\phi$.
\end{theorem}
The fact that the law of the excluded middle fails for Dependence Logic makes this formalism quite safe from the Liar Paradox. It is certainly possible to find a sentence $\lambda$ of Dependence Logic such that 
\[
	M_\omega \models^+ \tau(\ulcorner \lambda\urcorner) \Leftrightarrow M_\omega \models^+ \lambda \Leftrightarrow M_\omega \models^+ \lnot \tau(\ulcorner \lambda\urcorner) \Leftrightarrow  M_\omega \models^- \tau(\ulcorner \lambda \urcorner)
\]
however, this expression $\lambda$ will simply not be true in any model of Peano's Axioms, and therefore neither will $\tau(\ulcorner \lambda \urcorner)$ or its negation.

What about validity? 

Due to the equivalence between Dependence Logic and Existential Second Order Logic, the validity problem for Dependence Logic is highly undecidable:
\begin{theorem}
\label{theo:DL_val}
The Validity Problem of Dependence Logic is $\Pi_2$-complete with respect to the Levy Hierarchy.\footnote{In brief, this is the class of the set-theoretic statements of the form $\forall x \exists y \phi(x, y)$, where $\phi$ only contains bounded quantifiers. See \cite{levy65} for details.}
\end{theorem}
Many important set-theoretic statements belong in this class, and therefore correspond to Dependence Logic sentences: for example, one can find a Dependence Logic sentence which is valid if and only if there are no inaccessible cardinals (see Exercise 7.9 of \cite{vaananen07})!

As a consequence, there is no hope whatsoever of finding a semidecidable, sound and complete proof system for Dependence Logic. However, \cite{kontinenv11} describes a sound and complete Natural Deduction system for finding the \emph{first order} consequences of a Dependence Logic expression; and \cite{galliani13b} introduces a sound and complete axiom system for an extension of Dependence Logic, but with respect to a weaker semantics, analogous to Henkin's general semantics for Second Order Logic.

We barely scratched the surface of the current body of knowledge about Dependence Logic here, of course. But these few results are perhaps sufficient to point out some of the characteristics of this formalism, and will provide an useful basis of comparison for the paraconsistent version of it that we will develop in the next section.
\section{Paraconsistent Non-dependence Logic}
\label{sect:PDL}
In this section, we will introduce a novel, simple, paraconsistent variant of Dependence Logic. As we will see, many of the known results about Dependence Logic can be transferred to this new formalism with no difficulty whatsoever; and furthermore, as we will see, Paraconsistent Non-dependence Logic will be able to express not only its truth predicate, but also its \emph{validity} predicate!
\subsection{Paraconsistent Team Semantics}
As we saw, the Game Theoretic Semantics of Dependence Logic is a variant of that of First Order Logic in which a a player can only employ strategies which respect certain \emph{uniformity conditions}, specified in terms of dependence atoms. Then, as in the first order case, one characterizes the truth or falsity of Dependence Logic sentences in terms of the existence of uniform winning strategies for the existential or the universal players. 

However, one might wonder if this is condition is not too strict. If the semantic games represent debates between agents about the truth of a sentence, and if, as we saw, there exist some possible debates in which neither agent can guarantee a victory, perhaps one might want to be more permissive and say that a sentence is true if one's opponent cannot reliably succeed in \emph{falsifying} it? This idea would correspond to the following truth definition:
\begin{definition}
For all models $M$, teams $X$ and formulas $\psi$, 
\begin{itemize}
\item $(M, X) \pmodels^+ \psi$ if and only if $\abelard$ has \emph{no} winning strategy in $G^M_X(\psi)$; 
\item $(M, X) \pmodels^- \psi$ if and only if $\eloise$ has \emph{no} winning strategy in $G^M_X(\psi)$.
\end{itemize}
In particular, given a sentence $\phi$, we say that $\phi$ is \emph{paraconsistently true} in $M$, and we write $M \pmodels^+ \phi$, if and only if $(M, \{\emptyset\}) \pmodels^+ \phi$; and we say that it is \emph{paraconsistently false}, and we write $M \pmodels^- \phi$, if and only if $(M, \{\emptyset\}) \pmodels^- \phi$. 
\end{definition}
If our semantic games were determined, as the ones of First Order Logic are, this definition would be obviously equivalent to the usual one: indeed, if one of the players of a determined two-player game has no winning strategy, it follows at once that the other does. However, this is not the case for our games. For example, consider again the sentence $\phi = \forall x \exists y (=\!\!(\emptyset, y) \wedge x = y)$. As we saw, if $M$ has at least two elements then neither player has a winning strategy $G^M(\phi)$; and hence, we have at once that $M \pmodels^+ \phi$ and $M \pmodels^- \phi$. In other words, our $\phi$ is a simple example of a sentence which is \emph{both} true and false! 

More in general, it is easy to see that for all $M$, $X$ and $\psi$, 
\[
	(M, X) \pmodels^\pm \psi \Leftrightarrow (M, X) \not \models^\mp \psi.
\]

This implies at once that the resulting formalism, which we will call ``Paraconsistent Non-dependence Logic'' is contained in \emph{Team Logic}, that is, in the extension of Dependence Logic which adds a contradictory negation to it \cite{vaananen07b}. Furthermore, this property makes it trivial to find the Team Semantics rules for this logic
\begin{theorem}
\label{theo:parateam}
For all models $M$, teams $X$, relation symbols $R$ in its signature, tuples of terms $\tuple t$ and terms $t_1$, $t_2$ in its signature, variables $v$ and sets of variables $V$, 
\begin{description}
\item[PTS-atom$^+$] $(M, X) \pmodels^+ R \tuple t$ if and only if there exists a $s \in X$ such that $M \models_s R \tuple t$ in the usual first order sense; 
\item[PTS-eq$^+$] $(M, X) \pmodels^+ t_1 = t_2$ if and only if there exists a $s \in X$ such that $M \models_s t_1 = t_2$ in the usual first order sense;
\item[PTS-dep$^+$] $(M, X) \pmodels^+ =\!\!(\tuple t_1, \tuple t_2)$ if and only if $X \not = \emptyset$;
\item[PTS-neg$^+$] $(M, X) \pmodels^+ \lnot \psi$ if and only if $(M, X) \pmodels^- \psi$;
\item[PTS-$\vee^+$] $(M, X) \pmodels^+ \psi \vee \theta$ if and only if $(M, X) \pmodels^+ \psi$ or $(M, X) \pmodels^+ \theta$; 
\item[PTS-$\exists^+$] $(M ,X) \pmodels^+ \exists v \psi$ if and only if $(M, X[M/x]) \pmodels^+ \psi$.
\end{description}
and
\begin{description}
\item[PTS-atom$^-$] $(M, X) \pmodels^- R \tuple t$ if and only if there exists a $s \in X$ such that $M \not \models_s R \tuple t$ in the usual first order sense; 
\item[PTS-eq$^-$] $(M, X) \pmodels^- t_1 = t_2$ if and only if there exists a $s \in X$ such that $M \not \models_s t_1 = t_2$ in the usual first order sense;
\item[PTS-dep$^-$] $(M, X) \pmodels^- =\!\!(\tuple t_1, \tuple t_2)$ if and only if there exist two $s, s' \in X$ which assign the same values to $\tuple t_1$ but not to $\tuple t_2$;
\item[PTS-neg$^-$] $(M, X) \pmodels^- \lnot \psi$ if and only if $(M, X) \pmodels^+ \psi$;
\item[PTS-$\vee^-$] $(M, X) \pmodels^- \psi \vee \theta$ if and only if for all $Y$ and $Z$ such that $X = Y \cup Z$, $(M, Y) \pmodels^- \psi$ or $(M, Z) \pmodels^- \theta$;
\item[PTS-$\exists^-$] $(M, X) \pmodels^- \exists v \psi$ if and only if for all $F: X \rightarrow \parts(\dom(M))\backslash \{\emptyset\}$, $(M, X[F/v]) \pmodels^- \psi$.
\end{description}
\end{theorem}
Defining conjunction and universal quantification in the usual way, we have that 
\begin{description}
\item[PTS-$\wedge^+$] $(M, X) \pmodels^+ \psi \wedge \theta$ if and only if whenever $X = Y \cup Z$, $(M, Y) \pmodels^+ \psi$ or $(M, Z) \pmodels^+\theta$; 
\item[PTS-$\wedge^-$] $(M, X) \pmodels^- \psi \wedge \theta$ if and only if $(M, X) \pmodels^- \psi$ or $(M, X) \pmodels^- \theta$;
\item[PTS-$\forall^+$] $(M, X) \pmodels^+ \forall v \psi$ if and only if for all $F: X \rightarrow \dom(M)$, $(M, X[F/v]) \pmodels^+ \psi$; 
\item[PTS-$\forall^-$] $(M, X) \pmodels^- \forall v \psi$ if and only if $(M, X[M/v]) \pmodels^- \psi$. 
\end{description}
The rules that this semantics assigns to $=\!\!(\tuple t_1, \tuple t_2)$, however, do not model the notion of functional dependency in the same sense in which those of standard team-semantics do. Rather, they correspond to the notion of \emph{non-dependency}: $(M, X) \pmodels^- =\!\!(\tuple  t_1, \tuple t_2)$ if and only if $\tuple t_2$ is \emph{not} functionally dependent on $\tuple t_1$ in $X$.

Thus, we will define $\not = \!\!(\tuple t_1, \tuple t_2)$ as the formula $\lnot =\!\!(\tuple t_1, \tuple t_2)$ and we will think of it as a \emph{non-dependence atom}. Obviously, the satisfaction conditions this atom are obtained by switching the $+$ and $-$ signs in those of the dependence atom: in other words, they will hold positively in a team $X$ if and only if $\tuple t_2$ is \emph{not} a function of $\tuple t_1$ in $X$, and it will hold negatively in it only if $X \not = \emptyset$. 
\subsection{Properties}
Most of the properties of Dependence Logic are in natural correspondence with properties of Paraconsistent Non-dependence Logic. To begin with, it is clear that locality (in the sense of Theorem \ref{theo:local}) holds for Paraconsistent Non-dependence Logic too; and furthermore, 
\begin{theorem}
If $\psi$ is first order, 
\begin{description}
\item[PTS-FO$^+$] $(M, X) \pmodels^+ \psi$ if and only if there is a $s \in X$ such that $M \models_s \psi$ in the usual first order sense; 
\item[PTS-FO$^-$] $(M, X) \pmodels^- \psi$ if and only if there is a $s \in X$ such that $M \not \models_s \psi$ in the usual first order sense.
\end{description}
\end{theorem}
As a consequence, we have that if $\phi$ is a first order sentence, $M \models^+ \phi$ in Team Semantics if and only if $M \pmodels^+ \phi$ in Paraconsistent Team Semantics if and only if $M \models \phi$ in Tarski's semantics. Hence, like Dependence Logic proper, Paraconsistent Non-dependence Logic is a \emph{conservative} extension of First Order Logic.

As we saw, in Dependence Logic there exist no sentence $\phi$ and model $M$ such that $M \models^+ \phi \wedge \lnot \phi$; but there exist models $M$ and Dependence Logic sentences $\phi$ such that $M \not \models^+ \lnot (\phi \wedge \lnot \phi)$. On the other hand, there certainly exist sentences $\phi$ and models $M$ such that $M \pmodels^+ \phi \wedge \lnot \phi$; and furthermore, 
\begin{proposition}
For all models $M$, formulas $\psi$ in the signature of $M$ and nonempty teams $X$ with $\free(\psi) \subseteq \dom(X)$, 
\[
	(M, X) \pmodels^+ \lnot (\psi \wedge \lnot \psi)
\]
\end{proposition}
\begin{proof}
Clearly, $(M, X) \pmodels^+ \lnot ( \psi \wedge \lnot \psi)$ if and only if $(M, X) \pmodels^- \psi \wedge \lnot \psi$, that is, if and only if $(M, X) \pmodels^- \psi$ or $(M, X) \pmodels^+ \psi$. Suppose that neither of these is the case: then we have that $(M, X) \models^+ \psi$ and $(M, X) \models^- \psi$, which is impossible because of Theorem \ref{theo:lnc}.
\end{proof}
This is indeed a peculiar situation. Dependence Logic, which admits truth gaps, does not satisfy either the Law of Excluded Middle $\phi \vee \lnot \phi$ nor the Law of Non-Contradiction $\lnot (\phi \wedge \lnot \phi)$; and Paraconsistent Non-dependence Logic, which admits truth gluts, satisfies \emph{both}! The key for understanding this, of course, is to notice that the semantic games for Paraconsistent Non-dependence Logic formulas are not determined and, hence, our semantics for the negation in terms of player swapping - while being completely natural from a game-theoretic perspective - does not correspond to the contradictory negation: in general, asserting $\lnot \phi$ in Dependence Logic or in Paraconsistent Non-dependence Logic is not at all the same as rejecting $\phi$, although it reduces to it in the case of first order expressions.

Just as Dependence Logic satisfies the Downwards Closure property of Theorem \ref{theo:dc}, Paraconsistent Non-dependence Logic satisfies the following \emph{Upwards Closure Property:} 
\begin{theorem}
\label{theo:uc}
For all models $M$, teams $X, Y$ over it with $X \subseteq Y$, formulas $\psi$ over the signature of $M$ with $\free(\psi) \subseteq \dom(X)$ and $d \in \{+, -\}$, 
\[
	(M, X) \pmodels^d \psi \Rightarrow (M, Y) \pmodels^d \psi
\]
\end{theorem}
\begin{proof}
This can be proved very easily by means of Theorem \ref{theo:dc}, or, alternatively, by induction on $\psi$. However, we will show yet another proof, which uses the game-theoretic definition of truth and falsity in Paraconsistent Non-dependence Logic. Suppose that Player $\alpha \in \{\abelard, \eloise\}$ has no uniform winning strategy in $G^M_X(\psi)$. Then, since $X \subseteq Y$ he or she also has no uniform winning strategy in $G^M_Y(\psi)$: indeed, if a uniform strategy won this game for $\psi$ starting from all initial assignments in $Y$, he or she could also win starting from all assignments in $X \subseteq Y$. 
\end{proof}
Similarly, we have that 
\begin{proposition}
For all models $M$, formulas $\psi$, and $d \in \{+,-\}$, $(M, \emptyset) \not \pmodels^d \psi$.
\end{proposition}
\begin{proof}
Let $+^* = -$ and $-^* = +$. Now, by Proposition \ref{propo:emptyset}, we know that $(M, \emptyset) \models^{d^*} \psi$; but $(M, \emptyset) \pmodels^d \psi$ if and only if $(M, \emptyset) \not \models^{d^*} \psi$, and hence it follows at once that $(M, \emptyset) \not \pmodels^d \psi$.
\end{proof}
Also, as easy consequences of Theorems \ref{theo:DLsent} and \ref{theo:DLtoSigma11}, we have that 
\begin{theorem}
\label{theo:PDLsent}
For any sentence $\phi$ of Paraconsistent Non-dependence Logic there exists an universal second order sentence $\Phi$ such that $M \pmodels^+ \phi$ if and only if $M \models \Phi$. Conversely, any universal second order sentence $\Phi$ is logically equivalent to some Paraconsistent Non-dependence Logic sentence.
\end{theorem}
\begin{theorem}
\label{theo:PDLtoPi11}
For all formulas $\psi$ of Paraconsistent Non-dependence Logic and all sets of variables $V = \{v_1 \ldots v_n\}$ with $\free(\psi) \subseteq V$ there exists a $\Pi_1^1$ sentence $\Phi(R)$, where $R$ is a new $|V|$-ary relation symbol, such that 
\[
	(M, X) \pmodels^+ \psi \Leftrightarrow M \models \Phi(X(V))
\]
for all models $M$ and teams $X$. Furthermore, $R$ occurs only positively in $\Phi(R)$. 
\end{theorem}
Similarly, we have the following consequences of Theorems \ref{theo:DL_neg_nonsem} and  \ref{theo:Sigma11toDL}: 
\begin{theorem}
\label{theo:PDL_neg_nonsem}
Let $\psi_1$ and $\psi_2$ be two formulas such that $(M, X) \pmodels^+ \psi_1 \vee \psi_2$ for all models $M$ over the signature of $\psi_1$ and $\psi_2$ and for all nonempty teams $X$ with $\free(\psi_1) \cup \free(\psi_2) \subseteq \dom(X)$. Then there exists an expression $\theta$ such that 
\[
	(M, X) \pmodels^+ \theta \Leftrightarrow (M, X) \pmodels^+ \psi_1
\]
and 
\[
	(M, X) \pmodels^- \theta \Leftrightarrow (M, X) \pmodels^+ \psi_2
\]
for all models $M$ whose signature contains the ones of $\psi_1$ and $\psi_2$ and for all $X$ with domain containing $\free(\psi_1) \cup \free(\psi_2)$.
\end{theorem}
\begin{proof}
Let $\psi_1$ and $\psi_2$ be as in our hypothesis, and consider $\theta_1 = \lnot \psi_1$ and $\theta_2 = \lnot \psi_2$. Now, suppose that $(M, X) \models^+ \theta_1$ and $(M, X) \models^+ \theta_2$: then $(M, X) \not \pmodels^+ \psi_1$ and $(M, X) \not \pmodels^+ \psi_2$, and hence $(M, X) \not \pmodels^+ \psi_1 \vee \psi_2$, and hence $X = \emptyset$. Therefore, by Theorem \ref{theo:DL_neg_nonsem}, there exists a formula $\theta$ such that $(M, X) \models^+ \theta$ if and only if $(M, X) \models^+ \theta_2$ and $(M, X) \models^- \theta$ if and only if $(M, X) \models^+ \theta_1$. 

Then
\begin{align*}
	(M, X) \pmodels^+ \theta & \Leftrightarrow (M, X) \not \models^- \theta \Leftrightarrow (M, X) \not \models^+ \theta_1 \Leftrightarrow\\
	&\Leftrightarrow  (M, X) \pmodels^- \theta_1 \Leftrightarrow (M, X) \pmodels^+ \psi_1
\end{align*}
and 
\begin{align*}
	(M, X) \pmodels^- \theta & \Leftrightarrow (M, X) \not \models^+ \theta \Leftrightarrow (M, X) \not \models^+ \theta_2 \Leftrightarrow\\
	&\Leftrightarrow (M, X) \pmodels^- \theta_2 \Leftrightarrow (M, X) \pmodels^+ \psi_2
\end{align*}
as required.
\end{proof}

\begin{theorem}
\label{theo:Pi11toPDL}
Let $\Phi(R)$ be a $\Pi_1^1$ sentence in which $R$ occurs only positively and has arity $n$, and let $V = \{v_1 \ldots v_n\}$ be a set of variables. Then there exists a Paraconsistent Non-dependence Logic formula $\psi$ with free variables in $V$ such that \[
	(M, X) \pmodels^+ \psi \Leftrightarrow M \models \Phi(X(V))
\]
for all models $M$ and teams $X$.
\end{theorem}
\begin{proof}
Let $\Phi(R)$ be a $\Pi_1^1$ sentence as in our hypothesis. Then $\lnot \Phi(R)$ is equivalent to some $\Sigma_1^1$ sentence in which $R$ occurs only negatively; and therefore, by Theorem \ref{theo:Sigma11toDL}, there exists a Dependence Logic formula $\theta$ such that 
\[
(M, X) \models^+ \theta \Leftrightarrow M \models \lnot \Phi(X(V)).
\]
Now take $\psi = \lnot \theta$: then 
\begin{align*}
	(M, X) \pmodels^+ \psi &\Leftrightarrow (M, X) \not \models^- \lnot \theta \Leftrightarrow (M, X) \not \models^+ \theta \Leftrightarrow\\	
& \Leftrightarrow M \not \models \lnot \Phi(X(V)) \Leftrightarrow M \models \Phi(X(V))
\end{align*}
as required.
\end{proof}
So far so good: Paraconsistent Non-dependence Logic is so tightly coupled with Dependence Logic that importing even sophisticated  Dependence Logic results into it is quite unproblematic. 

Is Paraconsistent Non-dependence Logic capable of expressing its own truth predicate, like Dependence Logic is? As we will now see, the answer is positive, and furthermore this truth predicate is a simple variant of the truth predicate for Dependence Logic:
\begin{theorem}
For any finite signature $\Sigma$ containing the binary operations $+$ and $\times$, there exists a Paraconsistent Non-dependence Logic formula $\tau(x)$ with signature $\Sigma$ such that for all models $M_{\omega}$ of Peano's Axioms with signature $\Sigma$ and for all sentences $\phi$, 
\[
	M_{\omega} \pmodels^+ \phi \mbox{ if and only if } M_{\omega} \pmodels^+ \tau(\ulcorner \phi\urcorner)
\]
where $\ulcorner \phi\urcorner$ represents a G\"odel numbering for $\phi$.
\end{theorem}
\begin{proof}
Let $\tau^*$ be the truth predicate for Dependence Logic of Theorem \ref{theo:DL_TP}, such that
\[
	M_{\omega} \models^+ \phi \mbox{ if and only if } M_{\omega} \models^+ \tau^*(\ulcorner \phi\urcorner)
\]
for all $\phi$. 

Furthermore, let $\nneg(x)$ be a term such that $\nneg(\ulcorner \phi\urcorner) = \ulcorner \lnot \phi \urcorner$; and finally, define $\tau(x) = \lnot \tau(\nneg(x))$. Then for all models $M_{\omega}$ of Peano's Axioms with signature $\Sigma$, 
\begin{align*}
	&M_\omega \pmodels^+ \tau(\ulcorner \phi\urcorner) \Leftrightarrow M_\omega \pmodels^- \tau^*(\nneg(\ulcorner \phi\urcorner)) \Leftrightarrow\\
	&\Leftrightarrow  M_\omega \not \models^+ \tau^*(\ulcorner \lnot \phi\urcorner) \Leftrightarrow M_\omega \not \models^+ \lnot \phi \Leftrightarrow\\
	&\Leftrightarrow M_\omega \not \models^- \phi \Leftrightarrow M_\omega \pmodels^+ \phi
\end{align*}
as required.
\end{proof}
Once again, a difficult and important result for Dependence Logic can be transferred to Paraconsistent Non-dependence Logic with very little trouble! The answer that Paraconsistent Non-dependence Logic gives to the Liar Paradox is, of course, the dialetheist one. It is possible to find an expression $\lambda$ such that $M_\omega \models^+ \tau(\lambda)$ if and only if $M_\omega \models^- \lambda$ for all models of Peano's Axioms $M_\omega$; and for all such models, we will simply have that $\lambda$ and $\tau(\lambda)$ are both true and false. 

Hence, Paraconsistent Non-dependence Logic is a paraconsistent, conservative extension of First Order Logic that can express its own truth predicate! 
\subsection{Validity in Paraconsistent Non-dependence Logic}
In this section, we will examine the following concept of validity in Paraconsistent Non-dependence Logic:
\begin{definition}[P-Validity]
$\\$
A sentence $\phi$ of Paraconsistent Non-dependence Logic is \emph{paraconsistently valid}, or \emph{P-valid}, if and only if $M \pmodels^+ \phi$ for all first order models $M$ whose signature contains that of $\phi$.
\end{definition}
It is easy to see that differently from the case of Dependence Logic, validity in Paraconsistent Non-dependence Logic is semidecidable and reducible to validity in First Order Logic. This follows at once by the equivalence between Paraconsistent Non-dependence Logic and $\Pi_1^1$: every Paraconsistent Non-dependence Logic sentence is logically equivalent to (and effectively translatable into) some sentence of the form $\forall \tuple X \Psi(\tuple X)$, where $\Psi$ is first order, and hence it is valid if and only if $\Psi(\tuple R)$ is valid for some new tuple of relation symbols $\tuple R$. 

What about a proof system for this formalism? Once again, known results from Dependence Logic come to our succor: in particular, we can reduce the problem of finding valid Paraconsistent Non-dependence Logic sentences to the problem of finding \emph{contradictory} Dependence Logic sentences, which can be solved using Kontinen and V\"a\"an\"anen's proof system! 
\begin{theorem}
A sentence $\phi$ is P-valid if and only if for all models $M$, $M \not \models \lnot \phi$ with respect to Dependence Logic, that is, if and only if $\lnot \phi \models \bot$ with respect to the Dependence Logic consequence relation. 
\end{theorem}
\begin{proof}
Trivial.
\end{proof}
Hence, in order to check whether a sentence $\phi$ is P-valid, it suffices to check whether $\lnot \phi$ entails $\bot$ with respect to the proof system of \cite{kontinenv11}.

Finally, let us check that Paraconsistent Non-dependence Logic can define its own \emph{validity} predicate too. The main ingredient of the proof will be the following lemma: 
\begin{lemma}
For any signature $\Sigma$ exists a first order sentence $\pi_\Sigma$, with signature $\Sigma \cup \{+, \times, P\}$, where $P$ is a new symbol, such that for all models $M_{\omega}$ of Peano's Axioms, $(M_\omega, P) \models \pi_{\Sigma}$ if and only if $P$ represents a provability predicate for Paraconsistent Non-Dependence Logic, in the sense that for all Paraconsistent Non-Dependence Logic sentences $\phi$ and all $p \in M_\omega$
\[
	(p, \ulcorner \phi\urcorner) \in P \Leftrightarrow p \mbox{ encodes a proof of } \phi.
\]
\end{lemma}
We will not show the proof of this lemma here, but it poses no real difficulties: in brief, we encode its axioms and rules of the axiom system for Dependence Logic of \cite{kontinenv11} in the definition of $\pi_{\Sigma}$ in such a way that $(p, \ulcorner \phi \urcorner) \in P$ if and only if $p$ represents a proof of $\lnot \phi \vdash \bot$. Note that $\pi_\Sigma$ only needs to be a \emph{first-order} sentence: informally speaking, this corresponds to the known fact that the problem of the validity of $\Pi_1^1$ sentences can be straightforwardly reduced to the problem of the validity of first-order sentences. 

Now we can prove the main result of this section:
\begin{theorem}
For any finite signature $\Sigma$ containing the binary operations $+$ and $\times$, there exists a Paraconsistent Non-dependence Logic formula $\rho(x)$ with signature $\Sigma$ such that for all models $M_{\omega}$ of Peano's Axioms with signature $\Sigma$ and for all sentences $\phi$ of Paraconsistent Non-dependence Logic, 
\[
	M_{\omega} \pmodels^+ \rho(\ulcorner \phi\urcorner) \Leftrightarrow \phi \mbox{ is P-valid}
\]
\end{theorem}
\begin{proof}
By the previous lemma we know that $\phi$ is valid if and only if whenever $M_{\omega}$ satisfies Peano's Axioms, 
\[
	M_{\omega} \models \forall P(\pi_{\Sigma} \rightarrow \exists p (Pp\ulcorner \phi\urcorner)).
\]
Now let $c$ be a new constant symbol, and consider the expression $\rho(c) = \forall P(\pi_{\Sigma} \rightarrow \exists p (Ppc))$. Since this is a sentence in $\Pi_1^1$, there exists a Paraconsistent Non-dependence Logic sentence $\rho(c)$ which is equivalent to it. Then $\rho(x)$ is the formula that we were searching: indeed, 
\begin{align*}
	M_\omega \models^+ \rho(\ulcorner \phi\urcorner) &\Leftrightarrow M_{\omega} \models \forall P (\pi_\Sigma \rightarrow \exists p (P p \ulcorner \phi\urcorner) ) \Leftrightarrow\\
	&\Leftrightarrow \phi \mbox{ is provable } \Leftrightarrow \phi \mbox{ is P-valid.}
\end{align*}
\end{proof}
Again, the fact that Paraconsistent Non-dependence Logic admits truth gluts deals easily with the resulting paradox of self-reference. It is possible to find a sentence $\nu$ that is valid if and only if $\lnot \rho (\ulcorner \nu\urcorner)$ is valid and, therefore, if and only if $\lnot \nu$ is valid; but this is not a problem, as then we simply have that $\nu$ and $\lnot \nu$ are both valid. 
\subsection{Consistent Sentences in Paraconsistent Non-dependence Logic}
In this subsection, we will attempt to to characterize the sentences of Paraconsistent Non-dependence Logic which are \emph{consistent}, in the sense given by the following definition:
\begin{definition}[Consistent Sentences]
A sentence $\phi$ is consistent if and only if there is no model $M$ such that $M \pmodels^+ \phi$ and $M \pmodels^- \phi$. 
\end{definition}
As we will see, a Paraconsistent Non-dependence Logic sentence is consistent if and only if it is equivalent to some First Order Logic sentence. This shows that Paraconsistent Non-dependence Logic is a very ``cautious'' sort of paraconsistent logic, in the sense that its consistent kernel is precisely First Order Logic itself. 

Our result is based on the \emph{flattening technique} from (\cite{vaananen07}, \S 3.5):
\begin{definition}[Flattening]
Let $\psi$ be a Paraconsistent Non-dependence Logic formula. Its \emph{flattening} $\psi^f$ is the first order formula defined recursively as follows:
\begin{enumerate}
\item If $\psi$ is a first order atom, $\psi^f = \psi$;
\item For all dependence atoms $=\!\!(\tuple t_1, \tuple t_2)$, $(=\!\!(\tuple t_1, \tuple t_2))^f = \top$;
\item $(\lnot \psi)^f = \lnot (\psi)^f$;
\item $(\psi \vee \theta)^f = \psi^f \vee \theta^f$; 
\item $(\exists v \psi)^f = \exists v (\psi^f)$.
\end{enumerate}
\end{definition}
\begin{proposition}
For all models $M$, teams $X$ over $M$, formulas $\psi$ of Paraconsistent Non-dependence Logic over the signature of $M$ and with free variables in $X$, and $d \in \{+,-\}$,
\[
	(M, X) \pmodels^d \psi^f \Rightarrow (M, X) \pmodels^d \psi.
\]
\end{proposition}
\begin{proof}
This can be proved by induction with little difficulty. Alternatively, we may recall from Proposition 3.40 of \cite{vaananen07} that if $(M, X) \models^d \psi$ then $(M, X) \models^d \psi^f$ and that, by definition, $(M, X) \pmodels^d \psi$ if and only if $(M, X) \not \models^{d^*} \psi$ and $(M, X) \pmodels^d \psi^f$ if and only if $(M, X) \not \models^{d^*} \psi^f$. The result then follows at once. 
\end{proof}
We now have all that we need to prove our main result for this subsection: 
\begin{theorem}
Let $\phi$ be a Paraconsistent Non-dependence Logic sentence that is consistent. Then $\phi$ is logically equivalent to $\phi^f$. 
\end{theorem}
\begin{proof}
We already know that for all models $M$ and $d \in \{+,-\}$, if $M \pmodels^d \phi^f$ then $M \pmodels^d \phi$. Hence, we only need to verify the converse implication. 

Suppose that $M \not \pmodels^+ \phi^f$: then, since $\phi^f$ is first order, $M \pmodels^- \phi^f$, and therefore $M \pmodels^- \phi$, and therefore - since $\phi$ is consistent - $M \not \pmodels^+ \phi$. Similarly, if $M \not \pmodels^- \phi^f$ then $M \pmodels^+ \phi^f$, and therefore $M \pmodels^+ \phi$ and $M \not \pmodels^- \phi$, as required.
\end{proof}

\section{Generalized Dependencies in Paraconsistent Team Semantics}
Even though functional dependency was the first dependency notion which was studied in the context of Team Semantics, it was by no means the only one. Many other notions of dependence or independence atoms have been studied in this context, such as for example the \emph{independence atoms} \cite{gradel13}
\begin{description}
\item[TS-ind] $(M, X) \models^+ \tuple t_1 ~\bot~ \tuple t_2$ if and only if $X(\tuple t_1 \tuple t_2) = X(\tuple t_1) \times X(\tuple t_2)$,
\end{description}
the \emph{inclusion} and \emph{exclusion} atoms \cite{galliani12}
\begin{description}
\item[TS-inc] $(M, X) \models^+ \tuple t_1 \subseteq \tuple t_2$ if and only if $X(\tuple t_1) \subseteq X(\tuple t_2)$, 
\item[TS-exc] $(M, X) \models^+ \tuple t_1 | \tuple t_2$ if and only if $X(\tuple t_1) \cap X(\tuple t_2) = \emptyset$, 
\end{description}
or the \emph{multivalued dependence atoms} \cite{engstrom12}
\begin{description}
\item[TS-mdep] $(M, X) \models^+ \tuple x \twoheadrightarrow \tuple y$ if and only if for every $s, s'\in X$ with $s(\tuple x) = s'(\tuple x)$ there exists a $s'' \in X$ such that $s''(\tuple x \tuple y) = s(\tuple x \tuple y)$ and $s''(z) = s'(z)$ for all variables $z \in \dom(X) \backslash \{\tuple x \tuple y\}$;
\end{description}
and the study of the resulting logics and their relationship (in particular, of the \emph{definability} relations between such dependency notions) is a currently very active area of research. As always, there is little trouble transferring the resulting findings from the case of team semantics to the case of paraconsistent team semantics; for instance, since we know from \cite{galliani12} that all nonempty NP properties of teams correspond to the satisfaction conditions of formulas of Independence Logic (that is, of First Order Logic with team semantics plus independence atoms), it follows at once that all co-NP properties of teams correspond to satisfaction conditions of formulas for the paraconsistent dual of this logic. 

In this section, we will focus on the case of \emph{Paraconsistent Dependence Logic}, that is, the logic obtained by adding directly the functional dependency atom $=\!\!(\tuple t_1, \tuple t_2)$ of Team Semantics to the Paraconsistent Team Semantics rather than transforming it into its dual $\not = \!\!(\tuple t_1, \tuple t_2)$. As we will see, every sentence of Paraconsistent Dependence Logic is logically equivalent to some sentence of First Order Logic.

Thus, functional dependence atoms - or, as we will see, any other first-order, downwards-closed dependency conditions - do not increase the expressive power of First Order Logic with Paraconsistent Team Semantics. It is important, however, to keep in mind that this result applies to \emph{sentences} and not to formulas: indeed, it is easy to see that the functional dependence atom $=\!\!(\emptyset, x)$ is not logically equivalent (in Paraconsistent Team Semantics) to any formula of Paraconsistent Dependence Logic (and, in particular, to any first order formula), since it is not upwards closed in the sense of Theorem \ref{theo:uc}. Nonetheless, any sentence which uses such an atom can be converted into one which does not use it.

Let us begin with some definitions: 
\begin{definition}[Dependencies and First Order Dependencies]
Let $k \in \mathbb N$. A \emph{dependency} $\Dep$ is a family, closed under isomorphisms, of models over the signature $\{P\}$, where $P$ is a $k$-ary predicate symbol. The corresponding satisfiability conditions in Team Semantics and Paraconsistent Team Semantics are given by
\begin{description}
\item[TS-$\Dep^+$:] $(M, X) \models^+ \Dep \tuple t$ if and only if $(\dom(M), X(\tuple t)) \in \Dep$; 
\item[TS-$\Dep^-$:] $(M, X) \models^- \Dep \tuple t$ if and only if $X = \emptyset$;
\item[PTS-$\Dep^+$] $(M, X) \pmodels^+ \Dep \tuple t$ if and only if $(\dom(M), X(\tuple t)) \in \Dep$;
\item[PTS-$\Dep^-$] $(M, X) \pmodels^- \Dep \tuple t$ if and only if $X \not = \emptyset$.\footnote{For simplicity's sake, we limit ourselves to assigning non-trivial conditions to \emph{positive} satisfiability. However, there are no great difficulties in generalizing this approach to the case in which nontrivial positive and negative satisfaction conditions. The reason why a negated dependence atom is assumed to be true only in the empty set in the case of Team Semantics and to be true in all teams \emph{except} the empty set in the case of Paraconsistent Team Semantics is that we want to maintain the correspondence between team semantics and game-theoretic semantics insofar as it is possible, and either player has - trivially - a winning strategy for all games starting from the empty set of winning strategies.}
\end{description}
Such a dependency is said to be \emph{first order} if there exists a first-order sentence $\Dep^*$ such that 
\[
M \in \Dep \Leftrightarrow M \models \Dep^*
\]
for all models $M = (A, S)$ over the signature $\{P\}$. Given a family of dependencies $\DD$, we write $\FO(\DD)$ and $\PFO(\DD)$ for the logics obtained by adding all these dependencies to the language of first-order logic (with Team Semantics or Paraconsistent Team Semantics, respectively).
\end{definition}

Even if $\Dep$ is a first order dependency, there is in general no guarantee that $\FO(\Dep)$ and $\PFO(\Dep)$ are no more expressive than First Order Logic. For instance, all the non-dependency atoms $\not =\!\! (\tuple t_1, \tuple t_2)$ are first order, for $\Dep^* = \exists \tuple x \tuple y_1 \tuple y_2 (P \tuple x \tuple y_1 \wedge P \tuple x \tuple y_2 \wedge \tuple y_1 \not = \tuple y_2)$; and, as we saw, all $\Pi_1^1$ properties are expressible in terms of such dependencies (and of the language of First Order Logic, interpreted according to Paraconsistent Team Semantics). The same can be said about functional dependency atoms in team semantics. 

\begin{definition}[Downwards and Upwards Closed Dependencies]
A dependency $\Dep$ is said to be \emph{downwards closed} if and only if 
\begin{equation*}
 (A, S) \in \Dep, S' \subseteq S \Rightarrow (A, S') \in \Dep
\end{equation*}
for all models $(A, S)$ over the signature $\{P\}$. Similarly, $\Dep$ is said to be \emph{upwards closed} if and only if 
\begin{equation*}
 (A, S) \in \Dep, S \subseteq S' \Rightarrow (A, S') \in \Dep.
\end{equation*}
\end{definition}

We will make use of the following result, from \cite{galliani13e}:
\begin{theorem}
Let $\DD$ be a family of first order, upwards closed dependencies. Then every sentence $\phi$ of $\FO(\DD)$ is positively logically equivalent to some first order sentence $\phi'$, in the sense that 
\begin{equation*}
M \models^+ \phi \Leftrightarrow M \models \phi'
\end{equation*}
for all suitable models $M$.\footnote{However, it is not in general guaranteed that the negation of $\phi$ in Team Semantics is equivalent to the negation of $\phi'$ in Tarski's semantics.}
\end{theorem}
\begin{corollary}
Let $\not =\!\!(\cdot, \cdot)$ represent the family of all non-dependency atoms. Then every sentence of $\FO(=\!\!(\cdot, \cdot))$ is equivalent to some first order sentence. 
\end{corollary}
\begin{proof}
It suffices to observe that all such atoms are first order and upwards closed.
\end{proof}

Another ingredient of our proof will be the following observation: 
\begin{lemma}
Let $=\!\!(\cdot, \cdot)$ represent the family of all functional dependency atoms and the family of all functional non-dependency atoms respectively. Then for any formula $\phi$ of $\PFO(=\!\!(\cdot, \cdot))$ there exists some formula $\phi^*$ of $\FO(\not =\!\!(\cdot, \cdot))$ such that 
\begin{equation*}
(M, X) \pmodels^\pm \phi  \Leftrightarrow (M, X) \not \models^\pm \phi^*.
\end{equation*}
\end{lemma}
\begin{proof}
Let $\phi^*$ be the formula obtained from $\lnot \phi$ by replacing all dependence atoms $=\!\!(\tuple t_1, \tuple t_2)$ with negations of the the corresponding non-dependence atoms $\lnot \not =\!\!(\tuple t_1, \tuple t_2)$. We show, by induction on $\phi$, that $\phi^*$ works as required: 
\begin{itemize}
\item For all first-order atoms $R \tuple t$, $(M, X) \pmodels^\pm R \tuple t \Leftrightarrow (M, X) \not \models^\mp R \tuple t \Leftrightarrow (M, X) \not \models^\pm \lnot R \tuple t$, as required. 
\item $(M, X) \pmodels^+ =\!\!(\tuple t_1, \tuple t_2)$ if and only if any two $s, s' \in X$ which coincide over $\tuple t_1$ also coincide over $\tuple t_2$. But this is the case if and only if $(M, X) \not \models^+ \not = \!\!(\tuple t_1, \tuple t_2)$, that is, if and only if $(M, X) \not \models^+ \lnot \lnot \not =\!\!(\tuple t_1, \tuple t_2)$, as required. 

On the other hand, $(M, X) \pmodels^- =\!\!(\tuple t_1, \tuple t_2)$ if and only if $X$ is nonempty; and that is the case if and only if $(M, X) \not \models^- \lnot \lnot \not = \!\!(\tuple t_1, \tuple t_2)$, since any generalized dependence atom is satisfied by a team $X$ in Team Semantics if and only if the team is empty.
\item $(M, X) \pmodels^\pm \lnot \phi$ if and only if $(M, X) \pmodels^\mp \phi$: and by induction hypothesis, this is the case if and only if $(M, X) \not \models^\mp \phi^*$, that is, if and only if $(M, X) \not \models^\pm (\lnot \phi)^*$.
\item $(M, X) \pmodels^+ \psi \vee \theta$ if and only if $(M, X) \pmodels^+ \psi$ or $(M, X) \pmodels^+ \theta$, that is, by induction hypothesis, if and only if $(M, X) \not \models^+ \psi^*$ or $(M, X) \not \models \theta^*$. This is the case if and only if $(M, X) \not \models^+ \psi^* \wedge \theta^*$, which by De Morgan's Laws\footnote{Since we defined conjunction and universal quantification using De Morgan's Laws, and due to the fact that our definition of negation respects the rule of double negation, these laws hold in both Team Semantics and Paraconsistent Team Semantics.} is easily seen to be equivalent to $(\psi \vee \theta)^*$.

Furthermore, $(M, X) \pmodels^- \psi \vee \theta$ if and only if whenever $X = Y \cup Z$, $(M, Y) \pmodels^- \psi$ or $(M, Y) \pmodels^- \theta$, that is, by induction hypothesis, $(M, Y) \not \models^- \psi^*$ or $(M, Z) \not \models^- \theta^*$. By the rules of Team Semantics, this is the case if and only if $(M, X) \not \models^- \psi^* \wedge \theta^*$, which is again equivalent to $(\psi \vee \theta)^*$. 
\item $(M, X) \pmodels^+ \exists v \phi$ if and only if $(M, X[M/v]) \pmodels^+ \phi$, that is, if and only if $(M, X[M/v]) \not \models^+ \phi^*$, that is, if and only if $(M, X) \not \models^+ \forall v \phi^*$. Observe now that $\forall v \phi^*$ is equivalent to $(\exists v \phi)^*$.

Furthermore, $(M, X) \pmodels^- \exists v \phi$ if and only if for all nondeterministic choice functions $F$ we have that $(M, X[F/v]) \pmodels^- \phi$, that is, if and only if for all such $F$ $(M, X[F/v]) \not \models^- \phi^*$. Thus, $(M, X) \not \models^- \forall v \phi^*$, and the conclusion follows.
\end{itemize}
\end{proof}

Putting everything together, we obtain at once that 
\begin{theorem}
Let $\phi$ be a sentence in $\PFO(=\!\!(\cdot, \cdot))$. Then $\phi$ is positively logically equivalent to some first order sentence. 
\end{theorem}
\begin{proof}
By the lemma, there is a sentence $\phi^* \in \FO(\not = \!\!(\cdot, \cdot))$ such that $M \pmodels^+ \phi$ if and only if $M \not \models^+ \phi^*$. But by the above corollary, we can find a first-order sentence $\phi'$ such that $M \models^+ \phi^*$ if and only if $M \models \phi'$. Thus, $M \pmodels \phi$ if and only if $M \models \lnot \phi'$, as required.
\end{proof}

This result, of course, can be extended with ease to all other first-order downwards-closed dependencies by observing that the negations of these dependencies are first-order \emph{upwards-closed} dependencies (and, hence, safe to add to Team Semantics).

The situation in Paraconsistent Team Semantics is, once again, the dual of that of Team Semantics: even though non-dependence atoms bring the expressive power of our formalism far beyond that of First Order Logic (for instance, there exists a $\Pi_1^1$-sentence - and, therefore, a Paraconsistent Non-dependence Logic sentence - which is true in a model if and only if this model is finite), dependence atoms do not do so. This suggests the possibility that, perhaps, Paraconsistent Team Semantics may prove itself a less computationally demanding, but equally natural, framework than standard Team Semantics for the study of dependency notions in an abstract setting; but further work will be necessary in order to investigate the possibilities and the limitations of this approach.
\section{Conclusion}
As we saw, Paraconsistent Non-dependence Logic (that is, the variant of Dependence Logic in which the truth of a sentence is not defined in terms of existence of winning strategies of the Proponent $\eloise$, but in terms of the \emph{lack} of winning strategies for the Opponent $\abelard$) is a paraconsistent, conservative extension of First Order Logic that is capable of expressing its own truth and validity predicates. Furthermore, finding whether a sentence of this formalism is valid is no harder than finding whether a first order formula is valid; and the consistent fragment of this logic is equivalent precisely to First Order Logic. Just as Dependence Logic itself, Paraconsistent Non-dependence Logic admits both a game theoretic semantics and a compositional Team Semantics, and can be thought of as a fragment of Team Logic \cite{vaananen07b}; and furthermore, it is expressively equivalent to Universal Second Order Logic $\Pi_1^1$. On the other hand, Paraconsistent Dependence Logic is no more expressive than First Order Logic insofar as definability of classes of models is considered.

Ultimately, the reason for our formalism's paraconsistency - and the reason why it is capable of expressing its own truth and validity predicates - lies in its interpretation of negation, which is inherently game-theoretical: to negate a sentence is not to dismiss the possibility that it is true, but rather it is to admit one's incapability of forcing an opponent to admit its truth. The author dares not make any claims about this being the ``right'' notion of negation in some absolute sense: but, nonetheless, it seems to be a semantical operation which is worth investigating further.

Similarly, we most assuredly make no claim that Paraconsistent Non-dependence Logic is the ``One True Paraconsistent Logic''. It is a formalism that has a natural definition and elegant theoretical properties, and its notion of ``truth as lack of guarantee of falsifiability'' appears to have some epistemological interest; but whether a logic is or is not appropriate depends very much on what we want to use it \emph{for}. The purpose that led to the development of Paraconsistent Non-dependence Logic was to find a game-theoretically motivated, paraconsistent, conservative extension of First Order Logic which is able to express its own truth predicate; and, through this, to point out a possible connection between Game Theoretic Semantics, Dependence Logic and paraconsistency. Further research may be useful in order to unveil the possibilities of this connection: in particular, as we saw, much of the work of recent years about generalized dependency notions in team semantics and their definability relations (see for instance \cite{gradel13, galliani12, galliani13e}) can be adapted to the setting of Paraconsistent Team Semantics with little to no difficulty, and it provides ample opportunities to study logics obtained by adding ``dependencies'' corresponding to further epistemic properties of teams (interpreted as sets of possible states). Furthermore, even though the proof system of \cite{kontinenv11} can be used for finding valid formulas of Paraconsistent Non-Dependence Logic, it would be interesting to develop directly a proof system for this logic (or variants thereof).

One question that we leave entirely open is the connection between Paraconsistent Non-dependence Logic and other known paraconsistent logics: in particular, due to some conceptual similarity between the formalisms, we conjecture that a close relationship exists between Paraconsistent Non-dependence Logic and Ja\'skowski's \emph{Discussive Logic} \cite{jaskowski69,dacosta95}.

In any case, Paraconsistent Non-dependence Logic appears to be an intriguing conservative extension of First Order Logic which adds to it the ability of expressing its own truth predicate; and it is the hope of the author that the results of this paper may contribute to a more vigorous exchange of ideas between the Paraconsistent Logic and the Dependence Logic research communities. 
\bibliographystyle{plain}
\bibliography{biblio}

\begin{thebibliography}{10}

\bibitem{abramsky09}
Samson Abramsky and Jouko V\"a\"an\"anen.
\newblock From {IF} to {BI}.
\newblock {\em Synthese}, 167:207--230, 2009.
\newblock 10.1007/s11229-008-9415-6.

\bibitem{beziau00}
Jean-Yves B{\'e}ziau.
\newblock What is paraconsistent logic.
\newblock In {\em In Batens et al.[6}. Citeseer, 2000.

\bibitem{cameron01}
Peter Cameron and Wilfrid Hodges.
\newblock {S}ome {C}ombinatorics of {I}mperfect {I}nformation.
\newblock {\em The Journal of Symbolic Logic}, 66(2):673--684, 2001.

\bibitem{dacosta95}
Newton C.~A. da~Costa and Francisco~A. Doria.
\newblock On jaśkowski's discussive logics.
\newblock {\em Studia Logica}, 54:33--60, 1995.

\bibitem{durand11}
Arnaud Durand and Juha Kontinen.
\newblock Hierarchies in dependence logic.
\newblock {\em CoRR}, abs/1105.3324, 2011.

\bibitem{enderton70}
Herbert~B. Enderton.
\newblock Finite partially-ordered quantifiers.
\newblock {\em Mathematical Logic Quarterly}, 16(8):393--397, 1970.

\bibitem{engstrom12}
Fredrik Engstr{\"o}m.
\newblock Generalized quantifiers in dependence logic.
\newblock {\em Journal of Logic, Language and Information}, 21(3):299--324,
  2012.

\bibitem{galliani11}
Pietro Galliani.
\newblock Sensible semantics of imperfect information.
\newblock In Mohua Banerjee and Anil Seth, editors, {\em Logic and Its
  Applications}, volume 6521 of {\em Lecture Notes in Computer Science}, pages
  79--89. Springer Berlin / Heidelberg, 2011.

\bibitem{galliani12}
Pietro Galliani.
\newblock Inclusion and exclusion dependencies in team semantics: On some
  logics of imperfect information.
\newblock {\em Annals of Pure and Applied Logic}, 163(1):68 -- 84, 2012.

\bibitem{galliani13}
Pietro Galliani.
\newblock Epistemic operators in dependence logic.
\newblock {\em Studia Logica}, 101(2):367--397, 2013.

\bibitem{galliani13e}
Pietro Galliani.
\newblock Upwards closed dependencies in team semantics.
\newblock In Gabriele Puppis and Tiziano Villa, editors, {\em Proceedings
  Fourth International Symposium on Games, Automata, Logics and Formal
  Verification}, volume 119 of {\em EPTCS}, pages 93--106, 2013.

\bibitem{galliani13b}
Pietro Galliani et~al.
\newblock General models and entailment semantics for independence logic.
\newblock {\em Notre Dame Journal of Formal Logic}, 54(2):253--275, 2013.

\bibitem{gradel13}
Erich Gr{\"a}del and Jouko V{\"a}{\"a}n{\"a}nen.
\newblock Dependence and independence.
\newblock {\em Studia Logica}, 101(2):399--410, 2013.

\bibitem{henkin61}
Leon Henkin.
\newblock {S}ome {R}emarks on {I}nfinitely {L}ong {F}ormulas.
\newblock In {\em Infinitistic Methods. Proc. Symposium on Foundations of
  Mathematics}, pages 167--183. Pergamon Press, 1961.

\bibitem{hintikka68}
Jaakko Hintikka.
\newblock {L}anguage-{G}ames for {Q}uantifiers.
\newblock In {\em {A}merican {P}hilosophical {Q}uarterly {M}onograph {S}eries
  2: {S}tudies in {L}ogical {T}heory}, pages 46--72. Basil Blackwell, 1968.

\bibitem{hintikka96}
Jaakko Hintikka.
\newblock {\em The Principles of Mathematics Revisited}.
\newblock Cambridge University Press, 1996.

\bibitem{hintikka83}
Jaakko Hintikka and Jack Kulas.
\newblock {\em The Game of Language: Studies in Game-Theoretical Semantics and
  Its Applications}.
\newblock D. Reidel Publishing Company, 1983.

\bibitem{hintikkasandu89}
Jaakko Hintikka and Gabriel Sandu.
\newblock {I}nformational independence as a semantic phenomenon.
\newblock In J.E Fenstad, I.T Frolov, and R.~Hilpinen, editors, {\em Logic,
  methodology and philosophy of science}, pages 571--589. Elsevier, 1989.

\bibitem{hintikkasandu97}
Jaakko Hintikka and Gabriel Sandu.
\newblock Game-{T}heoretical {S}emantics.
\newblock In Johan van Benthem and Alice~T. Meulen, editors, {\em Handbook of
  Logic and Language}, pages 361--410. Elsevier, 1997.

\bibitem{hodges97}
Wilfrid Hodges.
\newblock {C}ompositional {S}emantics for a {L}anguage of {I}mperfect
  {I}nformation.
\newblock {\em Journal of the Interest Group in Pure and Applied Logics}, 5
  (4):539--563, 1997.

\bibitem{jaskowski69}
Stanis{\l}aw Ja{\'s}kowski.
\newblock Propositional calculus for contradictory deductive systems.
\newblock {\em Studia Logica}, 24:143--157, 1969.
\newblock 10.1007/BF02134311.

\bibitem{kontinennu09}
Juha Kontinen and Ville Nurmi.
\newblock Team logic and second-order logic.
\newblock In Hiroakira Ono, Makoto Kanazawa, and Ruy de~Queiroz, editors, {\em
  Logic, Language, Information and Computation}, volume 5514 of {\em Lecture
  Notes in Computer Science}, pages 230--241. Springer Berlin / Heidelberg,
  2009.

\bibitem{kontinenv09}
Juha Kontinen and Jouko V\"a\"an\"anen.
\newblock On definability in dependence logic.
\newblock {\em {J}ournal of {L}ogic, {L}anguage and {I}nformation},
  3(18):317--332, 2009.

\bibitem{kontinenv10}
Juha Kontinen and Jouko V\"a\"an\"anen.
\newblock {A} {R}emark on {N}egation of {D}ependence {L}ogic.
\newblock {\em Notre Dame Journal of Formal Logic}, 52(1):55--65, 2011.

\bibitem{kontinenv11}
Juha Kontinen and Jouko V{\"a}{\"a}n{\"a}nen.
\newblock Axiomatizing first-order consequences in dependence logic.
\newblock {\em Annals of Pure and Applied Logic}, 164(11):1101--1117, 2013.

\bibitem{levy65}
Azriel L\'evy.
\newblock A hierarchy of formulas in set theory.
\newblock {\em Memoirs of the American Mathematical Society}, 57, 1965.

\bibitem{mann11}
Allen~L. Mann, Gabriel Sandu, and Merlijn Sevenster.
\newblock {\em Independence-Friendly Logic: A Game-Theoretic Approach}.
\newblock Cambridge University Press, 2011.

\bibitem{priest79}
Graham Priest.
\newblock The logic of paradox.
\newblock {\em Journal of Philosophical Logic}, 8:219--241, 1979.
\newblock 10.1007/BF00258428.

\bibitem{priest06}
Graham Priest.
\newblock {\em In Contradiction}.
\newblock Oxford University Press, Oxford, 2006.

\bibitem{priest04}
Graham Priest, JC~Beall, and Bradley Armour-Garb (eds).
\newblock {\em The Law of Non-Contradiction: New Philosophical Essays}.
\newblock Oxford University Press, Oxford, 2004.

\bibitem{priest08}
Graham Priest and Francesco Berto.
\newblock Dialetheism.
\newblock In Edward~N. Zalta, editor, {\em The Stanford Encyclopedia of
  Philosophy}. Summer 2010 edition, 2010.

\bibitem{vaananen07}
Jouko V\"a\"an\"anen.
\newblock {\em Dependence Logic}.
\newblock Cambridge University Press, 2007.

\bibitem{vaananen07b}
Jouko V\"a\"an\"anen.
\newblock {T}eam {L}ogic.
\newblock In J.~van Benthem, D.~Gabbay, and B.~L\"owe, editors, {\em
  {I}nteractive {L}ogic. {S}elected {P}apers from the 7th {A}ugustus de
  {M}organ {W}orkshop}, pages 281--302. {A}msterdam {U}niversity {P}ress, 2007.

\bibitem{vaananen10}
Jouko V\"a\"an\"anen and Wilfrid Hodges.
\newblock Dependence of variables construed as an atomic formula.
\newblock {\em Annals of Pure and Applied Logic}, 161(6):817 -- 828, 2010.

\bibitem{walkoe70}
Wilbur~John Walkoe.
\newblock Finite partially-ordered quantification.
\newblock {\em The Journal of Symbolic Logic}, 35(4):pp. 535--555, 1970.

\end{thebibliography}
\end{document}